\documentclass[11pt]{article}
\usepackage{amsmath, amsthm,graphicx,amsfonts,amssymb,mathrsfs,showidx}
\usepackage{relsize}
\usepackage{color}
\definecolor{rojo}{rgb}{1,0,0}
\usepackage[all]{xy}
\makeindex
\usepackage{hyperref}
\usepackage{color}
\definecolor{abelian}{cmyk}{0.50,0,1,.4}
\definecolor{noabelian}{cmyk}{0.94,0.54,0,0}
\definecolor{rojo}{cmyk}{0,1,1,0}
\definecolor{verde}{cmyk}{0.91,0,0.88,0.12}
\xyoption{arc}

\newcommand{\Abe}{\textcolor{abelian}}
\newcommand{\noAbe}{\textcolor{noabelian}}

\newtheorem{thm}{Theorem}
\newtheorem{cor}[thm]{Corollary}
\newtheorem{lem}[thm]{Lemma}
\newtheorem{prop}[thm]{Proposition}

\theoremstyle{definition}
\newtheorem{defn}[thm]{\textbf{Definition}}

\theoremstyle{definition}

\theoremstyle{remark}

\newcommand{\op}{\operatorname}

\newcommand{\ben}{\begin{equation}}
\newcommand{\een}{\end{equation}}
\newcommand{\bena}{\begin{equation*}}
\newcommand{\eena}{\end{equation*}}

\newcommand{\ma}{\mathcal}

\newcommand{\ZZ}{\mathbb{Z}}

\def\ZZ{\mathbb{Z}}

\title{An approximation for the number of subgroups}
\author{Bruno Cisneros\thanks{IMUNAM-Oaxaca} and Carlos Segovia\thanks{IMUNAM-Oaxaca}}
\begin{document}
\maketitle

\begin{abstract}
Previously the second author has constructed by cobordism methods, an invariant associated to a finite group $G$. This invariant approximates the number of subgroups of a group, giving in some cases the number of abelian and cyclic subgroups. Here we explain the formulas used to obtain this invariant and we present values for some families of groups. 
\end{abstract}

\section*{Introduction}

An unsolved problem in geometric group theory is to give an explicit formula for the number of subgroups of a (finite) non-abelian group, even for abelian groups, this is a complicated task \cite{tarna,cald}. There are families of groups which admit a nice description of this number. For example the number of divisors of the integer $n$, denoted by $\tau(n)$, is the number of subgroups of the cyclic group $\ZZ_n$. For the Dihedral group $D_{2n}$, Stephan A. Cavior in \cite{cavior}, proved that the number of subgroups is given by $\tau(n)+\sigma(n)$, where $\sigma(n)$ is the sum of the divisors. Similarly, for the dicyclic groups $Dic_n$, the number of subgroups of $Dic_n$ coincides with $\tau(2n)+\sigma(n)$. For all the groups of order less than thirty, a list of the groups and the number of their subgroups is presented by G. A. Miller in \cite{Miller}.

    Hereafter comprise the results obtained with the help of the computer algebra system SageMath \cite{sage} and the group theory software \cite{gap}. 
We present a table containing for each group $G$ the number of subgroups, the number of abelian subgroups and the number $r(G)$ associated. 
The number $r(G)$ is written as a finite sum $r_1(G)+r_2(G)+\cdots$ and we report the first summand $r_1(G)$. 
We will prove in this work some important properties for the invariant $r(G)$.
 Indeed, for $G$ an abelian group, we obtain $r_i(G)=0$ for $i\geq 2$ and $r_1(G)$ has an explicit form as the cardinality of the quotient of $\{(k,g):k,g\in G\textrm{ and }[k,g]=1\}$ by an action of the special linear group $\op{SL}(2,\ZZ)$, which is partially contained in \cite{Kaufmann1}. For the cyclic group we obtain $r(\ZZ_n)=\tau(n)$ which is the number of divisors of $n$ and for $G$ the dihedral $D_{2n}$ and the dicyclic $Dic_n$ the number $r_1(G)$ is the number of abelian subgroups. For torsion groups $\ZZ_{p^n}$, the number is $r(\ZZ_{p^n})=\frac{p^{2n-1+p^{n+1}-p^{n-1}+p^2+p-1}}{p^2-1}$. For example for $p=2$ this is the sequence A007581 in \cite{oeis}.

\newpage
\bena\label{tabla} \textrm{LIST OF THE GROUPS AND THE NUMBER OF THEIR SUBGROUPS} \eena

\begin{tabbing}
  ORDs \=  DESCRIPTION OF THE GR\= UBGROUPS  \= lLIANGRPSxxx   \=  SAGExxxxx  \= SAGExx \= SAGE \kill
  \tiny{ORDERS} \> \tiny{DESCRIPTION OF THE GROUPS} \> \tiny{SUBGROUPS} \>  \tiny{ABE-SUBGROUPS} \>  \tiny{$r_1(G)$}   \\
  4 \> \Abe{Cyclic($\ZZ_4$)}, \Abe{$\ZZ_2^2$} \> 3,5 \>  3,5 \>       3,5 \\
  6 \> \Abe{$\ZZ_6$}, \noAbe{symmetric($\Sigma_3$)} \> 4,6 \> 4,5   \>     4,5\\
  8 \> \Abe{$\ZZ_8$}, \noAbe{octic($D_8$)}, \noAbe{quaternion($Q_8$)} \> 4,10+2,6+1 \> 4,9,5  \>     4,9,5 \\
   \> \Abe{$\ZZ_4\times\ZZ_2$}, \Abe{$\ZZ_2^3$} \> 8,16 \>8,16  \>    8,15\\
  9 \> \Abe{$\ZZ_9$}, \Abe{$\ZZ_3^2$} \> 3,6 \>3,6  \>        3,7 \\
  10 \> \Abe{$\ZZ_{10}$}, \noAbe{dihedral($D_{10}$)} \> 4,8 \> 4,7    \>      4,7  \\
  12 \> \Abe{$\ZZ_{12}$}, \noAbe{tetrahedral($A_4$)}, \noAbe{$D_{12}$} \> 6,10,16 \>6,9,13 \>        6,9,13  \\
   \>  \noAbe{Dicyclic($Dic_3$)}, \Abe{$\ZZ_2^2\times\ZZ_3$} \> 8,10 \> 7,10  \>       7,10 \\
  14 \> \Abe{$\ZZ_{14}$}, \noAbe{$D_{14}$} \> 4,10 \> 4,9  \>       4,9   \\
  15 \> \Abe{$\ZZ_{15}$} \> 4 \>  4  \>        4 \\
  16 \> \Abe{$\ZZ_{16}$}, \noAbe{$Dic_4$}, \noAbe{$D_{16}$}, \noAbe{$Q_8\times\ZZ_2$}  \> 5,11,19,19 \>   5,8,16,14  \>         5,8,16 \\
   \> \Abe{$\ZZ_8\times\ZZ_2$}, \Abe{$\ZZ_4^2$}, \Abe{$\ZZ_4\times \ZZ_2^2$} , \Abe{$\ZZ_2^4$} \> 11,15,27,67 \>  11,15,27,67 \>        11,16,25,51   \\
   \>  \noAbe{Modular group of order 16} \> 11 \>  10\>         10  \\
   \> \noAbe{Quasihedral of order 16} \> 15 \>  12 \>       12  \\
   \> \noAbe{$D_8\times\ZZ_2$} \> 35 \>  30  \>     28   \\
   \> \noAbe{$(\ZZ_4 \times\ZZ_2) \rtimes \ZZ_2$} \> 23 \>  22     \>      21  \\
   \>  \noAbe{$\ZZ_4\rtimes\ZZ_4 \textrm{ or }G_{4,4}$}\> 15 \>  14   \>        14  \\
   \>  \noAbe{$Q_8\rtimes\ZZ_2$}\> 23 \> 18   \>     18   \\
  18 \> \Abe{$\ZZ_{18}$}, \Abe{$\ZZ_3\times\ZZ_6$}, \noAbe{$D_{18}$} \> 6,12,16 \>  6,12,12   \>      6,14,12   \\
   \> \noAbe{$(\ZZ_3\times\ZZ_3)\rtimes\ZZ_2$}, \noAbe{$\Sigma_3\times\ZZ_3$} \> 28,14 \>  15,12 \> 16+,13  \\
  20 \> \Abe{$\ZZ_{20}$}, \Abe{$\ZZ_{10}\times\ZZ_2$}, \noAbe{$D_{20}$} \> 6,10,22 \> 6,10,19 \>          6,10,19   \\
   \>  \noAbe{$Dic_5$}, \noAbe{metacyclic}  \> 10,14 \> 9,12   \>        9+2,12   \\
  21 \> \Abe{$\ZZ_{21}$}, \noAbe{$\ZZ_7\rtimes\ZZ_3$} \> 4,10 \>  4,9    \>         4,9  \\
  22 \> \Abe{$\ZZ_{22}$}, \noAbe{$D_{22}$} \> 4,14 \> 4,13     \>     4,13\\
  24 \> \Abe{$\ZZ_{24}$},\,\Abe{$\ZZ_2\times\ZZ_{12}$},\,\Abe{$\ZZ_2\times\ZZ_2\times\ZZ_6$} \> 8,16,32 \> 8,16,32\>          8,16,30  \\
   \> \noAbe{$D_8\times\ZZ_3$},\,\noAbe{$Q_8\times\ZZ_3$} \> 20,12 \> 18,10     \>        18,10   \\
   \> \noAbe{$Sl(2,3)$}, \noAbe{$A_4\times\ZZ_2$} \> 15,26 \>  13,24  \>        13,23  \\
   \> \noAbe{$\Sigma_4$}, \noAbe{$D_{24}$}, \noAbe{$Dic_6$} \> 30,34,18 \>21,24,12  \>      21,24,12   \\
   \> \noAbe{$\ZZ_2\times\ZZ_2\times \Sigma_3$}, \noAbe{$\ZZ_2\times (\ZZ_3\rtimes\ZZ_4)$} \> 54,22 \> 43,19  \>        40,19     \\
   \> \noAbe{$\ZZ_4\times\Sigma_3$},\,\noAbe{$\ZZ_3\rtimes\ZZ_8$} \> 26,10 \>  21,9  \>     21,9   \\
   \> \noAbe{$(\ZZ_6 \times \ZZ_2) \rtimes \ZZ_2$} \> 30 \>  22     \>      22  \\
  25 \>  \Abe{$\ZZ_{25}$}, \Abe{$\ZZ_5^2$} \> 3,8    \>  3,8    \>      3,11  \\
  26 \>  \Abe{$\ZZ_{26}$}, \noAbe{$D_{26}$} \>  4,16 \>   4,15     \>        4,15  \\
  27 \> \Abe{$\ZZ_{27}$}, \Abe{$\ZZ_9\times\ZZ_3$}, \Abe{$\ZZ_3^3$} \> 4,10,28 \>   4,10,28   \>      4,12,40 \\
        \> \noAbe{$(\ZZ_3\times\ZZ_3)\rtimes\ZZ_3$}, \noAbe{$\ZZ_9\rtimes \ZZ_3$}  \> 19,10 \> 18,9   \>         22,10   \\
  28 \> \Abe{$\ZZ_{28}$}, \Abe{$\ZZ_{14}\times\ZZ_2$}, \noAbe{$D_{28}$}, \noAbe{$Dic_7$} \> 6,10,28,12  \>   6,10,25,11  \>      6,10,25,11   \\
  30 \> \Abe{$\ZZ_{30}$}, \noAbe{$D_{30}$}, \noAbe{$D_{10}\times\ZZ_3$}, \noAbe{$D_{6}\times\ZZ_5$}      \>     8,28,16,12 \>    8,19,14,10   \>      8,19,14,10 
\end{tabbing}

\section{Combinatorial description of this invariant}
\label{sec1}

We define the set $G(n)$ of {\bf minimal} sequences of $G$, with degree $n$, as follows
\bena
\left\{  (g_1,k_1,g_2,k_2,\cdots,g_n,k_n):\prod_{i=1}^n[k_i,g_i]=1\textrm{ and }\prod_{j=1}^r[k_j,g_j]\neq 1\neq\prod_{j=r+1}^n[k_j,g_j]\,\textrm{ for } 1\leq i <n-1\right\}\,.
\eena
The case $n=1$ is given by 
\bena G(1)=\{(k,g):[k,g]=1\}\,.\eena
Now identify these sequence as follows:
\begin{itemize}
\item[1)]take one index $i\in\{1,\cdots,n\}$ and identify the sequence $(g_1,k_1,\cdots,{\bf g_i,k_i},\cdots,g_n,k_n)$ with one of the possibilities,
\begin{itemize}
\item[a)] $(g_1,k_1,\cdots,{\bf g_i,k_ig_i^m},\cdots,g_n,k_n)$, for $m\in\ZZ$, and
\item[b)] $(g_1,k_1,\cdots,{\bf g_ik_ig_i^{-1},g_i^{-1}},\cdots,g_n,k_n)$;
\end{itemize}
\item[2)] for subsequent indices $i$ and $i+1$ we identify the sequence $(g_1,k_1,\cdots,{\bf g_i,k_i,g_{i+1},k_{i+1}},\cdots,g_n,k_n)$ with one of the possibilities,
\begin{itemize}
\item[a)] $(g_1,k_1,\cdots,{\bf [g_i,k_i]g_{i+1},k_{i+1},k_{i+1}g_ik_{i+1}^{-1},k_{i+1}k_ik_{i+1}^{-1}},\cdots,g_n,k_n)$, and 
\item[b)] $(g_1,k_1,\cdots,{\bf k_i^{-1}g_{i+1}k_i,k_i^{-1}k_{i+1}k_i,k_i^{-1}[k_i,g_i]k_ig_i,k_i},\cdots,g_n,k_n)$.
\end{itemize}
\end{itemize}
\begin{defn}
Consider the graph $\ma{G}(n)$ consisting of vertices $G(n)$ and edges given by the application of the preceding identifications. We define the positive number $r_n(G)$ as the number of connected components of $\ma{G}(n)$.
\end{defn}
Notice that for an abelian group $G$ the set $G(i)$ are empty for $i>1$, hence we obtain the next lemma.
\begin{lem}\label{l1} For $G$ abelian, we obtain $r_i(G)=0$, for $i>1$.
\end{lem}
For $n=1$ and $G$ not necessary abelian, there are two identifications in $G(1)$, 
\bena\label{in4}
(g,k)\sim(g,kg^m)\textrm{ for }m\in \ZZ
\textrm{ and  }(g,k)\sim(k,g^{-1})\,,
\eena 
since $[k,g]=1$. This translates into an action of the special linear group $\op{SL}(2,\ZZ_2)$ into the set $G(1)=\{(k,g):[k,g]=e\}$, by the application of the matrices
\bena\left(
\begin{array}{cc}
1&0\\ 1 & 1
\end{array}
\right)\textrm{   and   }\left(
\begin{array}{cc}
0&1\\ -1 & 0
\end{array}
\right)\,,\eena
respectively. Thus we obtain the following result.

\begin{prop} There is an action of the special linear group $\op{SL}(2,\ZZ)$ into the set $G(1)=\{(k,g):[k,g]=e\}$, whose number of orbits coincides with $r_1(G)$.
\end{prop}


An important component of the invariant $r(G)$ is given by the cyclic subgroups for the summand $r_1(G)$. In fact, two different cyclic subgroups will produce different classes with size given by the 
Jordan's function $J_k(n)$ (for $k=2$) which is a generalization of the Euler's function. 

\begin{prop}
Let $\op{Cyc}(G)$ be the number of cyclic subgroups of $G$ and let $\op{Com(G)}$ be the number of commuting pairs, i.e., $|\{(g,k):[k,g]=1\}|$. We have the inequality 
\ben
\op{Cyc}(G)\leq r_1(G)<\op{Com}(G)\,.
\een
\end{prop}

\section{The invariant for some families of groups}
\label{sec2}
Hereafter we find the explicit values for the cyclic group $G=\ZZ_n$, the dihedral $D_{2n}$, the dicyclic and torsion groups $\ZZ_{p^n}$.

\begin{thm}\label{teo1} The number $r_1(\ZZ_n)$ coincides with the number of subgroups of $\ZZ_n$.
\end{thm}
\begin{proof}
Set $(g,k)\in\ZZ_n\times\ZZ_n$, we can consider $k$ and $g$ positive numbers. By the algorithm axiom, there exist positive numbers $q_i$ and $r_j$, with $1\leq i\leq m$ and  $0\leq j\leq m$, such that
$r_0=k$, $r_1=g$, inequalities $0\leq r_m<\cdots <r_2<r_1=g$ and the equations
\bena
\begin{array}{l}
k=q_1g+r_2\\
g=q_2r_2+r_3\\
r_2=q_3r_3+r_4\\
\vdots\\
r_{m-2}=q_{m-1}r_{m-1}+r_m\\
r_{m-1}=q_mr_m\,,
\end{array}
\eena 
and we obtain the identifications 
\bena
\begin{array}{l}
(g,k)=(g,q_1g+r_2)\sim(g,r_2)\\
(g,r_2)=(q_2r_2+r_3,r_2)\sim (r_3,r_2)\\
(r_3,r_2)=(r_3,q_3r_3+r_4)\sim (r_3,r_4)\\
\vdots\\
\end{array}
\eena
ending in the pair $(r_m,0)$ or $(0,r_m)$. Notice for example that $(p,0)\sim(q,0)$ if and only if the generated subgroups by $p$ and $q$ are equal, i.e., $\langle p \rangle=\langle q \rangle$. Thus this subgroup generated by the residual element $r_m$ gives the corresponding subgroup of $\ZZ_n$ and the theorem follows. 
\end{proof}

\begin{thm}For the dihedral $D_{2n}$ and the dicyclic $Dic_n$, the numbers $r_1(D_{2n})$
 and $r_1(Dic_n)$ have the following values,
 \bena
 r_1(D_{2n})=
 \left\{\begin{array}{ll}
 \op{Cyc}(D_{2n})& \textrm{ if }n=2k+1\,,\\
 \op{Cyc}(D_{2n})+k&\textrm{ if }n=2k\,.
 \end{array}	
 \right.\textrm{ and }\hspace{0.5cm}
 r_1(Dic_{n})=\op{Cyc}(Dic_n)\,,
  \eena
  where $\op{Cyc}(D_{2n})=n+\tau(n)$, where $\tau(n)$ is the number of divisors of $n$.
 
 \end{thm}

\begin{cor}
For the dihedral $D_{2n}$ and the dicyclic $Dic_n$, the numbers $r_1(D_{2n})$
 and $r_1(Dic_n)$ coincide with the number of abelian subgroups of $D_{2n}$ and $Dic_n$, respectively.
\end{cor}

\begin{thm} Let $p$ be a prime number, for the $n$-product $\ZZ_p^n$, we obtain the equality \ben r_1(\ZZ_p^n)=\frac{p^{2n-1}+p^{n+1}-p^{n-1}+p^2-p-1}{p^2-1}\,.\een
\end{thm}
\begin{proof}
For $r_p^n:=r(\ZZ_p^n)$, let $F(n)$ be the number $r_p^{n+1}-r_p^n$. We will prove that
\ben\label{from1} F(n)=p^{n-1}(p^n+p-1)\,.\een Since
$r_p^{n}=(r_p^{n}-r_p^{n-1})+(r_p^{n-1}-r_p^{n-2})+...+(r_p^3-r_p^2)+(r_p^2-r_p^1)+r_p^1$,
where $r_p^1=2$ by Theo\-rem \ref{teo1}. Thus
$r_p^{n}=p^{n-2}(p^{n-1}+p-1)+p^{n-3}(p^{n-2}+p-1)+...+p(p^2+p-1)+(p+p-1)+2$
and as a consequence we have the following equations
\begin{align*}
r_p^{n}&=\sum_{i=0}^{n-2}p^{2i+1}+(p-1)\sum_{i=0}^{n-2}p^i+2\\
&=
p\frac{(p^2)^{n-1}-1}{p^2-1}+(p-1)\frac{p^{n-1}-1}{p-1}+2\\
&=\frac{p^{2n-1}-p+(p^{n-1}-1)(p^2-1)}{p^2-1}\\
&=\frac{p^{2n-1}+p^{n+1}-p^{n-1}+p^2-p-1}{p^2-1}\,.
\end{align*}

The formula \eqref{from1} follows by induction, applying the following identity
\ben\label{restante} F(n)=pF(n-1)+p^{2n-2}(p-1)\,.\een 
where $F(0)=r^1_p-r^0_p=2-1=1$. 
\end{proof}

\section{Cobordism methods}
\label{apen}

The invariant $r(G)$ associated to a finite group $G$, was defined in \cite{carla,carla2}, using cobordism methods. This number which was defined in a combinatorial way in section \ref{sec1}, 
represents the rank of the fundamental group of the 1+1 dimensional $G$-cobordism category. Although this complicated definition, this number describes the number of generators of the following monoid construction: consider all the principal $G$-bundles over connected, closed surfaces. Every $G$-bundle is given by a representation $\gamma$ of the fundamental group of $\Sigma_n$ consisting of elements (called monodromy) $g_i,k_i\in G$  ($i=1,\cdots,n$) satisfying $\prod_{i}[k_i,g_i]=1$. 
Identify two $G$-bundles if they are related by a diffeomorphism of the base which admits a lifting to the total space (this is called a $G$-equivariant diffeomorphism). 
Two $G$-bundles over the surfaces $\Sigma_{n_1}$ and $\Sigma_{n_2}$ respectively, admits a connected sum by considering an interior disk over each surface where both $G$-bundles are trivial, then choose a diffeomorphism between the boundaries of the disks and glue the bases spaces creating a surface $\Sigma_{n_1+n_2}$ with an induced $G$-bundle. This $G$-bundle is principal and moreover the class up to $G$-equivariant diffeomorphism is well-defined. The monoid conformed by these classes of principal $G$-bundles over the surfaces $\Sigma_n$ (with $n\geq 0$), with composition given by the connected sum, is abelian, without torsion and finitely generated (for a proof see \cite{carla2}). Consequently the number $r(G)$ writes as a sum $r_1(G)+r_2(G)+\cdots$, where for every summand $r_i(G)$, the subscript $i$ represents the genus of the generator.

The calculus of these numbers $r_i(G)$ for specific finite groups, uses some equations motivated by the axioms a $G$-Frobenius algebra \cite{turaev,kaufmann,a10}. 
We represent by the pair $(g,k)$ the monodromy of the two generators for a genus one surface, similarly, $(g_1,k_1,g_2,k_2)$ represent the consecutive pair of monodromy of a two genus  surface and hence, the sequence $(g_1,k_1,g_2,k_2,\cdots,g_n,k_n)$ represents the pairs of monodromies of a $n$-genus surface. 
Thus a method to find $r_i(G)$ consider sequence of pairs $(g_1,k_1,g_2,k_2,\cdots,g_i,k_i)$ (with $\prod_{i}[k_i,g_i]=1$), which are not written as a subsequent concatenation of sequences with trivial commutator product and then we apply the identification up to $G$-equivariant diffeomorphism. The use of 
Morse and Cerf theory \cite{Mil1,a11}, implies that every identification up to $G$-equivariant diffeomorphism reduces to some equations for genus one and two which was introduced in section \ref{sec1}. These equations depends on the application of the Dehn twist \cite{Benson}, the interchange of the two generators in a handle, see Figure \ref{f1}, and the clockwise and counterclockwise interchange of two adjacent genus, see Figure \ref{f2}.

\ben\label{f1}
\includegraphics[scale=1.1]{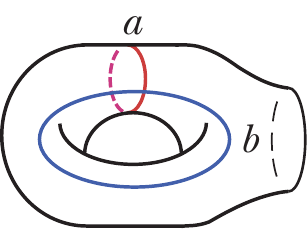}
\een

\ben\label{f2}
\includegraphics{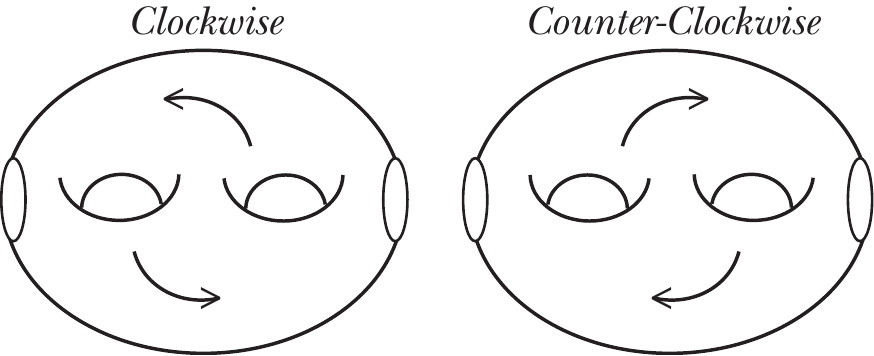}
\een
\newpage
\bibliographystyle{amsalpha}
\bibliography{biblio}

\end{document}